\documentclass{aptpub}
\usepackage{chicago}
\usepackage{graphicx}
\usepackage{amssymb}
\usepackage{amsmath}
\usepackage{subfigure}

\newcommand{\RR}{\mathbb R}
\newcommand{\GG}{\mathcal G}
\newcommand{\NN}{\mathbb N}
\newcommand{\CC}{\mathcal C}
\newcommand{\KK}{\mathcal K}

\authornames{S. Jaimungal, A. Kreinin, and A. Valov} 
\shorttitle{Randomized First Passage Times}

\begin{document}

\title{Randomized First Passage Times}

\authorone[Department of Statistics, University of Toronto]{Sebastian Jaimungal}
\addressone{100 St. George Street, Toronto, Ontario, Canada M5S 3G3}

\authortwo[Algorithmics Inc.]{Alex Kreinin}
\addresstwo{Algorithmics Inc., 185 Spadina Avenue, Toronto, Ontario, Canada M5T 2C6}

\authorone[Department of Statistics, University of Toronto]{Angelo Valov}

\begin{abstract}
In this article we study a problem related to the first passage and inverse first passage time problems for Brownian motions originally formulated by \citeN{JacksonKreininZhang2008}. Specifically, define $\tau_X = \inf\{t>0:W_t + X  \le b(t) \}$ where $W_t$ is a standard Brownian motion, then given a boundary function $b:[0,\infty) \to \RR$ and a target measure $\mu$ on $[0,\infty)$, we seek the random variable $X$ such that the law of $\tau_X$ is given by $\mu$. We characterize the solutions, prove uniqueness and existence and provide several key examples associated with the linear boundary.
\end{abstract}

\section{Introduction}

In this paper we examine a problem related to both the first-passage and inverse passage time problems. It was originally formulated by \citeN{JacksonKreininZhang2008} for the case of a linear boundary. In general, let $\{\Omega, \mathbb P, \mathbb F \}$ denote a complete probability space, and let $\mathbb F=\{\textsl{F}_t\}_{0\leq t\leq T}$ denote the natural filtration generated by the standard Brownian motion $W_t$. Consider a deterministic function $b:[0,\infty)\rightarrow \mathbb{R}$ with $b(0)>-\infty$ (this will represent the stopping boundary for the Brownian motion) and a random variable $X\geq b(0)$ (this will represent the randomization of the initial starting point of the Brownian motion). We will assume that $X$ spans the $\sigma$-algebra $\textsl{F}_0$ and is independent of the Brownian motion. Define the stopping time of the randomized Brownian motion hitting the boundary $b(t)$ as follows
\begin{eqnarray}
\tau_X \triangleq \inf\{\ t>0\ ;\ W_t+X \leq b(t) \ \}\ . \label{rfpt}
\end{eqnarray}
Without loss of generality we can take $b(0)=0$ so that $X$ is non-negative (see Figure \ref{randomFPT}).
Given the above, we are interested in the following problem:
\begin{defn}[Randomized First Passage Time Problem (RFPT)]
Given a boundary function $b:[0,\infty)\rightarrow \RR$,
and a probability measure $\mu$ on $[0,\infty)$, find a random variable $X$ such that $\mu$ is the law of the randomized first passage time $\tau_X$.
\end{defn}
\begin{figure}[t!]
\begin{center}
\label{randomFPT}
\includegraphics[scale=0.5]{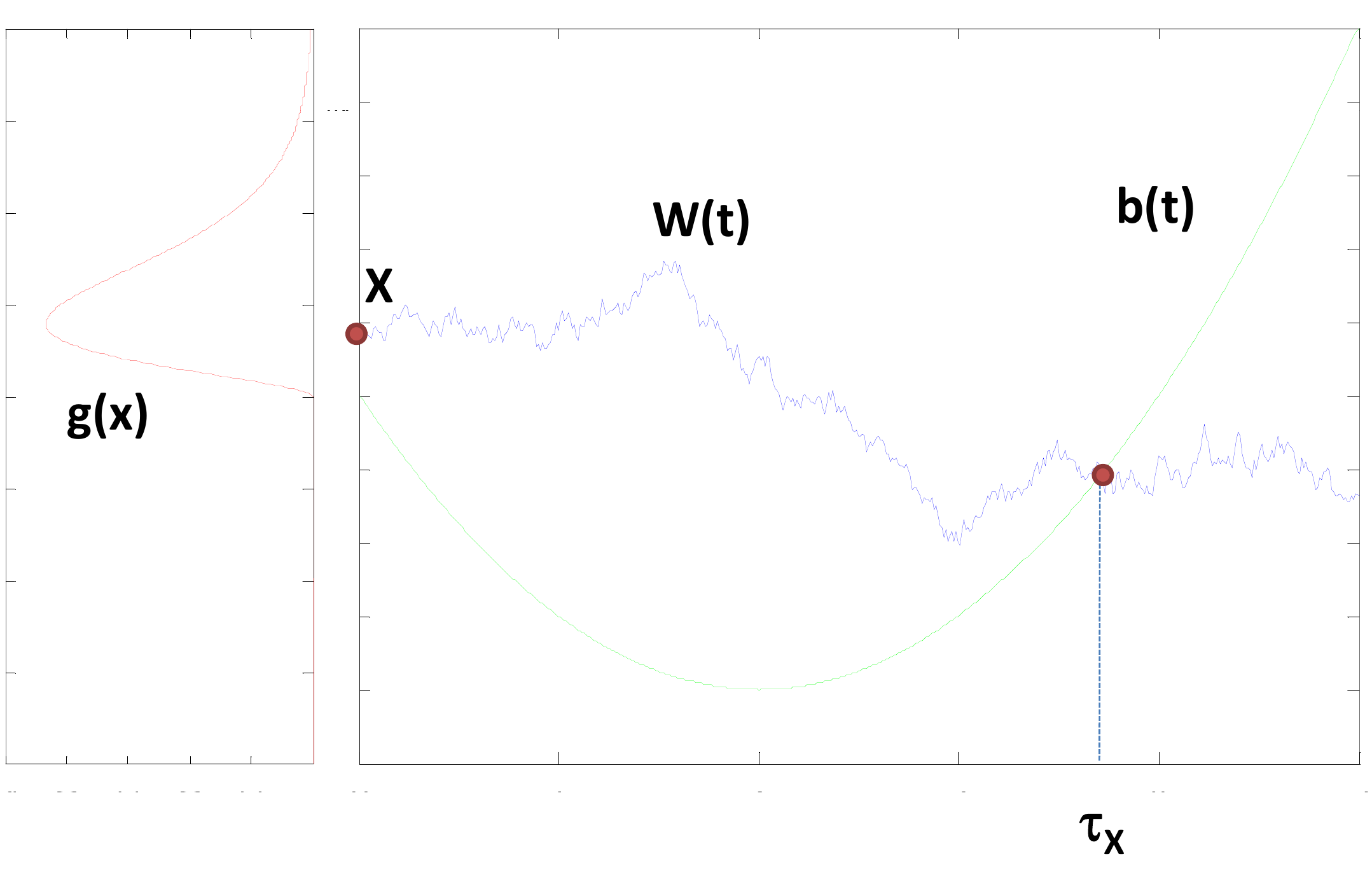}
\vspace{-1em}
\end{center}
\caption{A sample path of the randomized Brownian motion $W(t)+X$. $g(x)$ is the density of the starting point and $\tau_X$ is the first passage time of the randomized Brownian motion to the boundary $b(t)$. }
\end{figure}

Although the RFPT problem is stated in full generality, we will assume that $b(t)$ is continuously differentiable in order to ensure that the conditional distribution of $\tau|_{X=x}$ has a continuous density function $f(t|x)$ (see \citeN{Peskir1}). Furthermore, we will assume that $\mu$ is absolutely continuous with density function $f(t)$ which we refer to as the \textit{target density}. Then the RFPT probelm is equivalent to solving the Fredholm integral equation:
\begin{eqnarray}\label{eqn:Fredholm}
\int_0^{\infty}f(t|x)g(x)dx=f(t) \ . \label{eqn:matchmain}
\end{eqnarray}
Here, the conditional density of the hitting time $f(t|x)$ acts as the kernel in the integral equation. In general we seek solutions to (\ref{eqn:matchmain}) in the class of functions $\GG:=\{\ g \ ; \ \int_{R^+}|g|<\infty\ or\ |g|<K \ \}$; however, our main focus is on the class of density functions on the positive real line which is included in $\GG$.

In this paper we address the questions of existence (see Theorem \ref{existence}) and uniqueness (see Propositions \ref{prop:hermitetransform} and \ref{prop:laplacetransform}) of the distribution of the random initial point $X$. Theorem \ref{existence} is one of the main results of this work and it implies the existence of a random variable X with moment generating function $\tilde{g}$ given by
\begin{eqnarray}
\tilde{g}(\alpha)=\int_0^{\infty}e^{-\alpha b(t)-\alpha^2t/2}f(t)dt \ . \label{eqn:Result1}
\end{eqnarray}
This existence result is very powerful, and provides a relationship between the Laplace transform of the initial point and a boundary specific integral transform of the target density. For particular boundaries, this integral transform is of a standard type. For example, if the boundary is linear $b(t)=\mu t$, then the integral transform in the rhs of \eqref{eqn:Result1} is simply the Laplace transform at the point $\alpha \mu + \frac{1}{2}\alpha^2$. However, in general the integral transform may not be easily identifiable; nonetheless, one can in principle compute the integral transform numerically and then numerically invert $\tilde{g}$ to obtain the density.

Note that, since the boundary $b(t)$ defines the conditional distribution of the stopping time $f(t|x)$ uniquely, when $X$ is a discrete random variable the distribution of the unconditional stopping time will be a mixture of the conditional densities as seen from (\ref{eqn:Fredholm}). The reverse statement, however, does not hold. We show this in the case $b(t)=\mu t$. Given the well known fact that continuous distributions on the positive real line can be arbitrarily well approximated (at least point-wise) by a mixture of Gamma distributions (see e.g. \citeN{Tjims95}), we derive analytical formula for the density function of $X$ when the target distribution of $\tau_X$ is a finite mixture of Gamma distributions. This formula allows us to extend our analytical results for $g$ to a very large class of target distributions.


When $X$ is non-random, the RFPT problem reduces to the classical first-passage time (FPT) problem for Brownian motion if we take the boundary function $b(.)$ as input and seek the density function $f(.)$. The FPT problem has a long history and dates back to the work of A.N. Kolmogorov, A. Khinchin, I. Petrovsky and P. Levy. A good summary of the their early work can be found in \citeN{Khinchine33}. The available closed form results appear to be sparse, fragmentary and essentially confined to the linear, quadratic (see \citeN{Salminen88}) and square root (see \citeN{Novikov81}) boundaries. A unified approach for the derivation of the analytical results for these three cases was presented in \citeN{FPTpaper} based on the Fredholm equation
\begin{eqnarray}
\int_0^{\infty}e^{-\alpha b(t)-t\alpha^2/2}f(t)dt=1 \ ,\label{eqn:mainreal}
\end{eqnarray}
where $f$ is the density function of the first passage time of the Brownian motion to the boundary $b(t)$. Equation (\ref{eqn:mainreal}) holds for all $\alpha$  and continuous functions $b(.),\ b(0)>-\infty,$ satisfying
\begin{eqnarray}
\lim_{t \uparrow \infty}(b(t)+\alpha t) > -\infty \ .\label{fredholmcondition}
\end{eqnarray}
This condition simply states that the drifted boundary is uniformly bounded below.

\citeN{FPTpaper} also derive a class of Volterra Integral equations of first kind which generalizes and unifies all previous known FPT/IFPT integral equations. A key element in that work was the construction of the following new class of martingales.
\begin{prop}
The process
\begin{equation}
Z_s := m(s,W_s; p, t), \qquad m(s,w; p,t) = \frac{e^{-\frac{(w-y)^2}{4(t-s)} }  D_p\left( (w-y) / \sqrt{t-s}\right) }{ (t-s)^{(p+1)/2} } ,
\end{equation}
is a real valued martingale on $s\in[0,t)$ for all $p,y\in \mathbb R$, $t>0$. Here $D_p(z)$ is the parabolic cylinder function -- for some of its properties see Appendix \ref{sec:ParabolicCylinder}.
\end{prop}
Armed with this martingale, the authors then derive the class of Volterra integral equations
\begin{eqnarray}\label{matchvolterra}
m(t,0; p, t) = \int_0^t m(s,b(s); p,t) f(s) \ ds \ ,
\end{eqnarray}
for $y<b(t)$, by invoking the optional sampling theorem on the martingale $Z_s$. Furthermore, by passing to the limit $y \uparrow b(t)$, the authors show that the resulting class of Volterra equations contains the class of equations of \citeN{Peskir1} and show that the FPT density is the unique continuous solution to any member of this class under certain regularity conditions on the boundary.

When $X$ is non-random and we take the density function $f(.)$ as an input, the RFPT problem reduces to the inverse first-passage time (IFPT) problem which seeks the boundary function $b(.)$. The problem was first posed by A. Shiryaev in 1976 for the case of the exponential probability density. An early paper by \citeN{Anulova80} deals with the existence of some stopping times for a given distribution, however, these stopping times are not of the form (\ref{rfpt}) for some function $b$. One of the main contributions to the description of the inverse problem thus far, was provided by \citeN{Saunders1} where they demonstrate the existence of a unique viscosity solution for the IFPT problem, from a PDE perspective, and describe the small time behaviour of the boundary function. Note that in the context of the IFPT problem the integral equations (\ref{matchvolterra}) are nonlinear.

Randomizing the starting point of the Brownian motion allows us to bypass both the first passage time and inverse first passage time problems. This is achieved by assuming the pair $(b,f)$ as given while trying to match the density $f$ by randomizing with a density function $g$. Furthermore the randomization allows the distribution of the hitting time $\tau_X$, and thus of the stopped process $W_{\tau_X}$, to probe a much wider class of distributions than if the Brownian motion starts at a fixed point. For example, in the linear boundary case, with slope equal to one, we have the relation $W_{\tau_X}=\tau_X-X$. If we can imply the distribution of $X$ from any distribution of $\tau_X$, then we may be able to describe the resulting class of distributions for $W_{\tau_X}$. This is tantalizingly close to the statement of Skorohod's embedding problem (see \citeN{Skorohod}), where one seeks a stopping time $\tau$ such that the stopped Brownian motion $W_\tau$ has a given distribution. Here, our class of stopping times $\tau_X$ is generated by fixing the boundary $b$ and randomizing the starting point $X$, and through the connection $W_{\tau_X} = b(\tau_X)$ the distributions of the stopped Brownian motion and the stopping time is clear.

If we attack the RFPT problem by seeking a direct solution to (\ref{eqn:matchmain}) it may seem, at a first glance, that this could be a formidable task since the kernel of this Fredholm equation, $f(t|x)$, is unknown for most boundary functions. However, as we will discover below, the problem is simpler than both the FPT and IFPT problems and we obtain analytical and semi-analytical results for certain transforms of the matching distribution using the Volterra and Fredholm integral equations derived in \citeN{FPTpaper}.

The remainder of the paper is organized as follows. In Section 1, we assume the existence of a function $g$ which solves (\ref{eqn:matchmain}) and, under certain conditions on the boundary $b$, derive unique integral transforms of $g$. The integral transforms lead to uniqueness of the solution and provide us with a way to compute $g$ analytically. In Section 2, we address the question of existence of the random variable $X$. In Section 3, we look at the linear boundary case and compute $g$ analytically for a class of target distributions. Moreover, we present a number of examples for the pair $(g,f)$. In Section 4, we examine the effect of an affine boundary transformation on the target density. For the case of the linear boundary, we explore the relationship between the target densities corresponding to different slopes assuming the same density function $g$. Furthermore, we analyze the implications of a certain scaling property of the boundary and in the linear boundary case we look at an extension of the RFPT problem when both the intercept and the slope are random. Finally, we end with some concluding remarks in Section 5. A number of technical proofs are delegated to the appendix.

\section{Uniqueness}

The integral equations (\ref{matchvolterra}) take on a particularly useful form when $p$ is restricted to the integers. In this case, the parabolic cylinder function $D_n$ is related to the Hermite polynomials $H_n$ as follows ($n\in\mathbb Z$):
\begin{equation}
D_n(z)=2^{-n/2}e^{-z^2/4}H_n\left(z/\sqrt{2}\right)\ .
\end{equation}
This allows equations (\ref{matchvolterra}), for $p=n$, to be rewritten as:
\begin{eqnarray}
\frac{e^{-\frac{y^2}{2t}}H_n(-y/\sqrt{2t})}{t^{(n+1)/2}}=\int_0^t\frac{e^{-\frac{(b(s)-x-y)^2}{2(t-s)}}}{(t-s)^{(n+1)/2}}H_n\left(\frac{b(s)-x-y}{\sqrt{2(t-s)}}\right)f(s|x)ds \ , \label{eqn:general1}
\end{eqnarray}
for any $y<b(t)-x$. Here we have the equation in the RFPT form by using the conditional hitting density $f(t|x)$. Note that since the Hermite polynomials form a complete orthogonal basis in $L^2(\mathbb R, e^{-x^2})$ with respect to the standard normal distribution, \eqref{eqn:general1} allows for a unique series representation of the density function $g$ whenever it exists. Here we seek uniqueness for the larger class $\GG$ which contains, but is not limited to, the class of densities. This leads us to the following results.
\begin{prop}\label{prop:hermitetransform}
Suppose $b:[0,\infty)\mapsto \mathbb R$ is a continuous function and there exists a $t>0$ such that $b(t)>0$. Then, if (\ref{eqn:matchmain}) has a solution $g\in G \bigcap L^2(\mathbb R, e^{-x^2})$, it is unique and it is given by
\begin{eqnarray}
g(x)=\frac{1}{\sqrt{2 \pi}}\sum_{n=0}^{\infty}\frac{t^{n/2}a_n(t)}{2^n n!}H_n(x/\sqrt{2t}) \ , \label{eqn:hermitetransform}
\end{eqnarray}
for any $t>0$ such that $b(t)>0$, and where
\begin{eqnarray} a_n(t):=\int_0^t\frac{e^{-\frac{b(s)^2}{2(t-s)}}}{(t-s)^{(n+1)/2}}H_n\left(\frac{b(s)}{\sqrt{2(t-s)}}\right)f(s)ds \ .
\end{eqnarray}
\end{prop}
\begin{proof} See Appendix A $\boxdot$
\end{proof}

This solution need not be a true density function. However, in the case when $g$ is a density function then it is the unique solution to the randomized FPT problem. Interesting, the above series representation holds for all $t>0$ such that $b(t)>0$, yet the solution $g(x)$ is independent of the specific choice of $t$.

Next we examine the Laplace transform of $g$ using the Fredholm equation of the first kind (\ref{eqn:matchmain}). Using (\ref{eqn:matchmain}) we obtain the following result.
\begin{prop}\label{prop:laplacetransform}
Suppose $b:[0,\infty)\mapsto\mathbb R$ is continuous and satisfies condition \eqref{fredholmcondition}. Then, if (\ref{eqn:matchmain}) has a solution $g\in \GG$, it is unique and its Laplace
transform is
\begin{eqnarray}
\tilde{g}(\alpha)=\int_0^{\infty}e^{-\alpha b(t)-t\alpha^2/2}f(t)dt \ .\label{eqn:matchlaplace}
\end{eqnarray}
\end{prop}
\begin{proof}
Assume that equation \eqref{eqn:matchmain} has a solution $g\in \GG$. Under the condition \eqref{fredholmcondition}, equation \eqref{eqn:mainreal} holds for all $\alpha > 0$ and (conditional on $X=x$) we have
\[
\int_0^{\infty}e^{-\alpha b(t)-t\alpha^2/2}f(t|x)dt=e^{-\alpha x} \ .
\]
Multiply both sides by $g(x)$ and integrate out $x$. By Fubini's theorem, the order of integration can be exchanged since \[
\int_0^{\infty}\int_0^{\infty}|g(x)|e^{-\alpha b(t)-t\alpha^2/2}f(t|x)dtdx=\int_0^{\infty}e^{-\alpha x}|g(x)|dx<\infty\ .
\]
Finally, using (\ref{eqn:matchmain}), we obtain the Laplace transform of $g$ given in (\ref{eqn:matchlaplace}). Uniqueness then follows from the uniqueness of the Laplace transform. $\boxdot$
\end{proof}

As a demonstration of the applicability of the above result, suppose $b(t)=\sqrt{t}$ and let $f$ be the unconditional density. Then \eqref{eqn:matchlaplace} becomes
\begin{eqnarray*}
\tilde{g}(\alpha)=\int_0^{\infty}e^{-\alpha \sqrt{t}-t\alpha^2/2}f(t)dt \ ,
\end{eqnarray*}
for all $\alpha>0$. Multiplying both sides of this equation by $\alpha^{p-1},\ p>0$, and integrating out $\alpha$ we obtain
\begin{align*}
& \int_0^{\infty}\alpha^{p-1}\tilde{g}(\alpha)d\alpha  \\
& \qquad = \int_0^{\infty}\alpha^{p-1}\int_0^{\infty}e^{-\alpha \sqrt{t}-t\alpha^2/2}f(t)\ dtd\alpha\\
& \qquad = \int_0^{\infty}t^{-p/2}f(t)\int_0^{\infty}u^{p-1}e^{-u-u^2/2} \ dudt
= e^{1/4} \Gamma(p)D_{-p}(1)\int_0^{\infty}t^{-p/2}f(t) \ dt\\
& \qquad = e^{1/4}\Gamma(p)D_{-p}(1)\widehat{f}(1-p/2)\ .
\end{align*}
Here, the substitution $u=\alpha \sqrt{t}$ was used in the second equality, and $\widehat{f}$ is the {\it Mellin transform} of $f$. Thus the Mellin transform of the Laplace transform of $g$, denoted $\widehat{\widetilde{g}}$ is given by
\[
\widehat{\widetilde{g}}(p)=e^{1/4}\Gamma(p)D_{-p}(1)\hat{f}(1-p/2)\ ,
\]
provided that $\hat{f}(1-p/2)$ exists for a non-empty set of positive real values of $p$.

\section{Existence}

We saw from Propositions \ref{prop:hermitetransform} and \ref{prop:laplacetransform} that if the boundary is well behaved and there exists a solution $g\in \GG$ of \eqref{eqn:matchmain} then it is unique. Thus existence of a solution to the RFPT problem implies uniqueness if the boundary satisfies the hypothesis of Proposition \ref{prop:laplacetransform} since any density function $g$ belongs to the class $\GG$. The question of showing the existence of a unique matching density $g$ reduces to finding conditions under which (\ref{eqn:matchmain}) has a solution. Given $b(t)$, there may not exist a density $g$ for every density function $f$ satisfying (\ref{eqn:matchmain}). An illustrative counterexample of existence is provided by the choice $b(t)=0$ and $f(t)=\lambda e^{-\lambda t},\ \lambda>0$. Then \eqref{eqn:matchlaplace} reduces to $\tilde{g}(\alpha)=\tilde{f}(\alpha^2/2)=2\lambda/(2\lambda +\alpha^2)$. Consequently, $\tilde{g}$ is the Laplace transform of $g(x;\lambda)=\sqrt{2\lambda}\sin(x\sqrt{2\lambda})$. Thus clearly, $g\in G$ but it is not a probability density function.

A sufficient requirement for the existence of a density solution to \eqref{eqn:matchmain} can be constructed based on Picard's Criterion (see e.g. \citeN{PolyaninManzhirov08}, p.578-583). However, such a construction is difficult since the kernel function $f(t|x)$ is not, in general, known explicitly. The only notable exceptions being the linear, square-root and quadratic boundaries mentioned earlier. Furthermore we cannot guarantee that the solution is a density function. Consequently, we approach the question of existence in a probabilistic manner. To this end, we seek a random variable $X$ such that
\begin{eqnarray}
\mathbb{E}^{q}(f(t|X))=f(t) \ , \label{matchprob}
\end{eqnarray}
where the expectation is taken under a measure $q$ with support on the positive real line. In light of the Laplace transform given in \eqref{eqn:matchlaplace}, the question of existence is reduced to examining sufficient conditions under which this transform is a moment generating function of some random variable $X$. It is well known that moment generating functions are completely monotone\footnote{Recall that a completely monotone function $r$ has derivatives of all orders which satisfy $(-1)^n\frac{d^n}{d\alpha^n}r(\alpha)\geq 0$ for all $\alpha > 0$ and all non-negative integers $n\geq 0$.} (see \citeN{Feller71}). Thus, if the function $r:[0,\infty)\mapsto (0,\infty)$, defined as
\begin{align}
r(\alpha):=\int_0^{\infty}e^{-\alpha b(t)-\alpha^2t/2}f(t)dt \ ,
\end{align}
is completely monotone, then it is our candidate for a moment generating function. The following Lemma provides an alternative check point for checking complete monotonicity of $r(\alpha)$ and proves to be a useful tool.
\begin{lem}\label{completemonotone}
Suppose $b:[0,\infty)\mapsto\mathbb R$ is continuous, satisfies condition \eqref{fredholmcondition}, and
\begin{eqnarray}
\int_0^{\epsilon}e^{\beta \frac{b(t)}{\sqrt{t}}}f(t)dt < \infty,\ \forall \beta>0 \ .
\end{eqnarray}
Then $r(\alpha)$ is completely monotone if and only if the pair $(b,f)$ satisfies
\begin{eqnarray}
\int_0^{\infty}t^{n/2}e^{-\alpha b(t)-\alpha^2 t/2}H_n\left(\frac{b(t)+\alpha t}{\sqrt{2t}} \right)f(t)dt \geq 0 \label{monotonecheck}
\end{eqnarray}
for all $n\geq 1$ and $\alpha>0$.
\end{lem}
\begin{proof} See Appendix A $\boxdot$
\end{proof}

Using the complete monotonicity property of $r(\alpha)$, we can now derive sufficient conditions for the existence of $X$. The following Theorem describes the properties of the functions $b$ and $f$ which guarantee the existence of $X$.
\begin{theorem} \label{existence}
Suppose that
\begin{enumerate}
 \item $b(t):[0,\infty)\mapsto \mathbb R$ is continuous and satisfies condition (\ref{fredholmcondition})
 \item $r(\alpha)$ is completely monotone
 \item If
 \begin{align}
 \int_0^{\infty}e^{-\alpha b(t)-\alpha^2t/2}z(t)dt=0 \label{trivial}
 \end{align}
holds for all $\alpha> 0$, then $z(t)$ is identically zero
 \end{enumerate}
Then, there exists a random variable $X$ with m.g.f. given by $r(\alpha)$ such that $\tau_X$ has probability density function $f$.
\end{theorem}
\begin{proof} Since $r$ is completely monotone and $r(0)=1$, by Bernstein's theorem (see \citeN{Feller71} pp. 439), there exists a probability measure on $[0,\infty)$ with cumulative distribution function $q$ such that
\[
r(\alpha)=\int_0^{\infty}e^{-x\alpha}dq(x) \ .
\]
On the other hand since equation (\ref{eqn:mainreal}) holds for all $\alpha > 0$ then we have
\[
\int_0^{\infty}e^{-\alpha b(t)-\alpha^2t/2}f(t|x)dt=e^{-\alpha x}
\]
for all $\alpha> 0$. Taking integrals on both sides of the above equation with respect to the function $q$ and using Fubini's theorem we obtain
\begin{eqnarray*}
r(\alpha)=\int_0^{\infty} e^{-\alpha x}dq(x)&=&\int_0^{\infty}\int_0^{\infty}e^{-\alpha b(t)-\alpha^2t/2}f(t|x)dtdq(x)\\
&=&\int_0^{\infty}e^{-\alpha b(t)-\alpha^2t/2}\int_0^{\infty}f(t|x)dq(x)dt \ .
\end{eqnarray*}
The above relation implies that
\begin{eqnarray}
\int_0^{\infty}e^{-\alpha b(t)-\alpha^2t/2}\left(\int_0^{\infty}f(t|x)dq(x)-f(t)\right)dt=0
\ . \label{exist2}
\end{eqnarray}
Since \eqref{exist2} holds for all $\alpha > 0$, assumption 3 implies that
\[
\int_0^{\infty}f(t|x)dq(x)=f(t)\ .
\]
Furthermore, integrating  the above w.r.t. $t$ on $[0,\infty)$ we have that $\int_0^{\infty}dq(x)=1$ since $f$ is a proper density function. Therefore $q$ defines a proper distribution function. Consequently, the equality $\mathbb{E}^q(f(t|X))=f(t)$, where the distribution of $X$ is given by $q$, holds and there exists a solution to the RFPT problem (\ref{matchprob}). $\boxdot$
\end{proof}

The hypothesis of Theorem \ref{existence} imply the assumptions of Proposition \ref{prop:laplacetransform} and therefore existence implies uniqueness. The assumption that \eqref{trivial} has only the trivial solution is certainly satisfied in the case $b(t)=\mu t,\ \mu>0$ because, in this case, \eqref{trivial} implies that the Laplace transform of $z(t)$ is zero and thus $z(t)$ is identically zero. The assumption is also satisfied when $b(t)=\mu \sqrt{t}$. In this case, applying the Mellin transform on both sides of \eqref{trivial}, implies that the Mellin transform of $z$ is zero and thus $z$ is also zero. These two examples demonstrate that the class of boundaries for which \eqref{trivial} has only the trivial solution is non-empty. The complete monotonicity of the function $r$ is harder to check in general. However, the case $b(t)=\mu t$ and $f(t)$ is the Gamma$(a,b)$ density can be checked directly using the conditions of Lemma \ref{completemonotone}.
\begin{cor}\label{linboundexist}
Let $b(t)=\mu t,\ \mu>0,$ and $f(t)=\frac{t^{b-1}e^{-t/a}}{\Gamma(b)(a)^b}$ with $a\geq 2/\mu^2$. Then there exists a unique random variable $X$ such that $\tau_X$ has the probability density $f$.
\end{cor}
\begin{proof}
Clearly, $b(t)$ satisfies the hypothesis of Proposition \ref{prop:laplacetransform} which shows the uniqueness part. For proof of existence see Appendix A. $\boxdot$
\end{proof}
Corollary \ref{linboundexist} motivates the next section where we take a closer look at the linear boundary case.

\begin{figure}[t!]
\begin{center}
\includegraphics[scale=0.5]{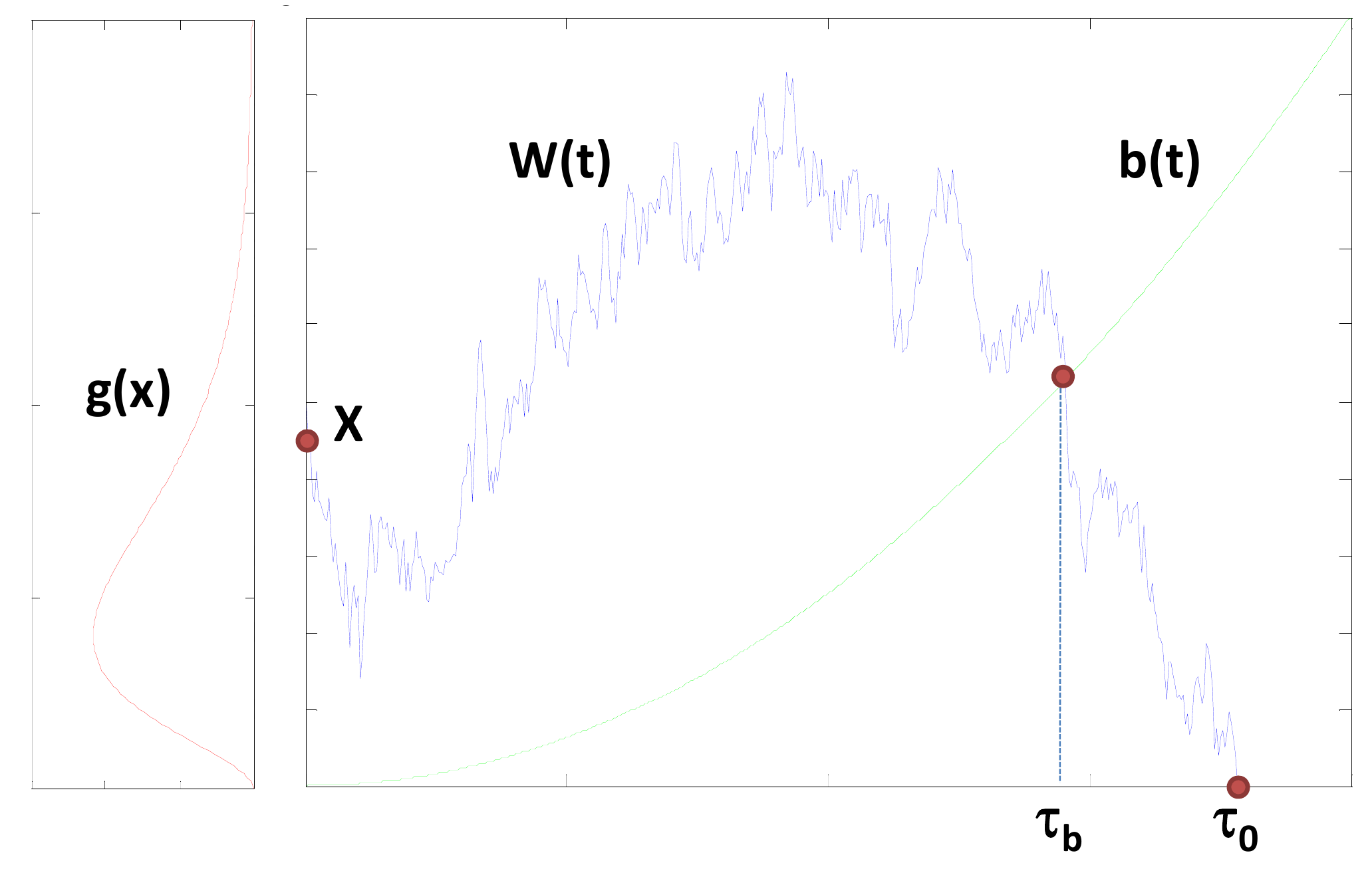}
\caption{ The probabilistic interpretation of equation \eqref{eqn:unconditionaly0v0}. The Brownian path must first hit the boundary $b$ before hitting zero. \label{fig:zeroHitting}}
\end{center}
\end{figure}
We end this section by showing that if there exists a unique solution $g$ to the RFPT problem for a given boundary $b(t)>0,\ \forall t>0,$ and target density $f_b(t)$, then $g$ is also a solution for the zero boundary with a particular target density $f_0$. To show this relationship, consider equation \eqref{eqn:general1}. When $n=1$ (and $y=-x$) and using $H_1(x)=2x$,  \eqref{eqn:general1} reduces to
\begin{eqnarray}
\int_0^t \frac{e^{-\frac{b(s)^2}{2(t-s)}}b(s)}{\sqrt{2 \pi}(t-s)^{3/2}}\ f(s|x)ds=\frac{e^{-x^2/2t} \ x}{\sqrt{2 \pi}t^{3/2}} \ .\label{eqn:generaly0v0}
\end{eqnarray}
The right side can be recognized as the probability density of the FPT of $x+W_t$ to the zero boundary. This observation admits a simple probabilistic interpretation of (\ref{eqn:generaly0v0}): the process $x+W_t$ first hits $b$ at time $\tau_b$ before it hits the zero boundary at time $\tau_0$ -- see Figure \ref{fig:zeroHitting}. Thus, if there exists a $g\in \GG$ which solves (\ref{eqn:matchmain}) for the boundary $b$ and unconditional density $f_b$, then
\begin{eqnarray}
\int_0^t \frac{e^{-\frac{b(s)^2}{2(t-s)}}b(s)}{\sqrt{2\pi}(t-s)^{3/2}}f_b(s)ds=f_0(t) \ , \label{eqn:unconditionaly0v0}
\end{eqnarray}
where $f_b(s)=\int_0^{\infty}f(s|x)g(x)dx$ and $f_0(t)=\int_0^{\infty}\frac{e^{-x^2/2t}x}{\sqrt{2\pi}t^{3/2}}g(x)dx$. If $b(t)>0$ for all $t>0$ then (\ref{eqn:generaly0v0}) and (\ref{eqn:unconditionaly0v0}) hold for all $t>0$. In fact  $f_0(t)$ is a proper density function, as can be seen by integrating the left side of (\ref{eqn:unconditionaly0v0}). Therefore,  $g \in GG$ is furthermore a density function, the corresponding unconditional distributions of the first passage times to $b(s)$ and to $0$ are related as in (\ref{eqn:unconditionaly0v0}). Moreover, \eqref{eqn:unconditionaly0v0} implies that for every distribution $f_b$, for which there is a matching distribution $g$, there exists a distribution $f_0$ (given by the integral in \eqref{eqn:unconditionaly0v0}) such that the pair $(f_0,b=0)$ has the same matching distribution $g$. As a result, the class of unconditional densities for the boundary $b(t)=0$ for which there exists a matching distribution is at least as large as the corresponding class of unconditional distributions for any boundary $b(t)>0$.

\section{Linear Boundary}

We now focus on the case when the boundary is linear and seek explicit solutions to the RFPT. When the starting position $X$ is non-random then the first passage time distribution of the Brownian motion to the linear boundary $bt-X$ is well known to be inverse Gaussian and is explicilty
\begin{eqnarray}
f_{\tau|X}(t)=\frac{X}{\sqrt{2 \pi t^3}}\exp\left\{-\frac{(bt-X)^2}{2t}\right\} \ . \label{inversegauss}
\end{eqnarray}
\citeN{JacksonKreininZhang2008} use this explicit form to demonstrate that the hitting time of a drifted Brownian motion with a random starting point can be Gamma distributed. In this section we corroborate this result based on our integral equation \eqref{eqn:matchlaplace} and extend it to a class of distributions which are infinite linear combinations of Gamma distributions.

Letting $b(t)=\mu t,\ \mu>0$ and denoting
\begin{eqnarray*}
\tau_{x,\mu}=\inf\{t>0;W_t\leq \mu t-x\} \ ,
\end{eqnarray*}
then, for $\alpha>-\mu$ (\ref{eqn:matchlaplace}) reduces to
\begin{eqnarray}
\tilde{g}(\alpha)=\int_0^{\infty}e^{-t(\alpha \mu+\alpha^2/2)}f(t)dt= \tilde{f}(\alpha \mu+\alpha^2/2) \ , \label{eqn:matchlinearb>0}
\end{eqnarray}
where $\tilde{f}$ is the Laplace transform of $f$, the distribution of the randomized stopping time $\tau_{X,\mu}$. When $f$ is the density of the Gamma distribution then we can factorize $\tilde{f}(\alpha \mu+\alpha^2/2)$ and write it as a product of two Laplace transforms of Gamma densities. As a consequence, $g$ is a convolution of Gamma distributions. The same argument applies when $f$ is a mixture of Gamma distributions. More formally, define the sequence $a_{n}$ such that $\inf_{n\geq 1}a_n\geq 2/\mu^2$ and the sequence $b_n>0,\ \forall n\in \NN$. Define the class of densities,
\begin{eqnarray}\label{linearclass}
\CC:=\left\{f \ : \ f(t)=\sum_{n=1}^{\infty}p_n f_{a_n,b_n}(t),\ \sum_{n=1}^{\infty}p_n=1,\ p_n\geq 0,\ \inf_{n\geq 1}a_n\geq 2/\mu^2\right\}\ ,
\end{eqnarray}
where $f_{a_n,b_n}$ are densities of Gamma distributions with scale parameters $a_n$ and shape parameters $b_n$. Then we have the following result:
\begin{thm}\label{mdlinbound}
Let $b(t)=\mu t,\ \mu> 0$. Suppose $\tau_{X,\mu}\sim f\in \CC$ . Then, the matching density $g(x)$ is given by
\begin{eqnarray}\label{gamma}
g(x)=\sum_1^{\infty} p_n\frac{\sqrt{2 \pi}e^{-\mu x}}{\Gamma(b_n)\sqrt{a_n}}\left(\frac{x}{a_n\sqrt{\mu^2-2/a_n}}\right)^{b_n-1/2}I_{b_n-1/2}(x\sqrt{\mu^2-2/a_n})  \ . \label{eqn:matchlinearmain}
\end{eqnarray}
\end{thm}
\begin{proof}
First we show that (\ref{eqn:matchlinearmain}) holds for a finite mixture of Gamma densitites. Let $f(t)=\sum_{n=1}^N p_nf_{a_n,b_n}(t), \ \sum_{n=1}^N p_n=1,\ p_n\geq 0$, where $f_{a_n,b_n}$ are Gamma densities with scale parameter $a_n$ and shape parameter $b_n$. Then, from (\ref{eqn:matchlinearb>0}), we have
\begin{align}
\tilde{g}(\alpha)=\tilde{f}(\mu\alpha + \alpha^2/2)&=\sum_{n=1}^{N} p_n (1+a_n(\mu\alpha +\alpha^2/2))^{-b_n}\nonumber \\
&=\sum_{n=1}^{N} p_n (1+c^+_n\alpha)^{-b_n}(1+c_n^-\alpha)^{-b_n} \label{eqn:proofGammaFac}
\end{align}
where $c_n^{\pm}=\frac{1}{2}a_n(\mu \pm \sqrt{b^2-2/a_n})$. For $\Re(c_n^{\pm})>0$ and $\Im(c_n^{\pm})=0$, i.e. $c_n^{\pm}$ to be positive real numbers, we require that $a_n\geq 2/\mu^2, \ n=1,...,N$. This restriction is already enforced by having $f\in\CC$. From \eqref{eqn:proofGammaFac}, $g(x)$ is a mixture of convolutions of Gamma r.v.'s and in particular
\begin{eqnarray*}
g(x)&=&\sum_1^N p_n \int_0^x f_{c_n^+,b_n}(u)f_{c_n^-,b_n}(x-u)du\\
&=& \sum_1^N p_n\frac{\sqrt{2 \pi}e^{-\mu x}}{\Gamma(b_n)\sqrt{a_n}}\left(\frac{x}{a_n\sqrt{\mu^2-2/a_n}}\right)^{b_n-1/2}I_{b_n-1/2}(x\sqrt{\mu^2-2/a_n}) \ ,
\end{eqnarray*}
where $I$ is the modified Bessel function of the first kind.

For an infinite mixture of gamma distributions the result now follows easily. For $\tau_{X,\mu}\sim f\in \CC$, substitute $f$ in (\ref{eqn:matchlaplace}) and using Fubini's theorem we can exchange the integration and summation (since all quantities are positive) to obtain the above result. The condition $\inf_{n\geq 1}a_n\geq 2/\mu^2$ ensures that each Laplace transform in the infinite mixture is factorizable with real valued roots. $\boxdot$
\end{proof}

Next we look at several simple examples for a finite and infinite Gamma density mixture.
\begin{itemize}

\item \textbf{Example 1:} For $N=1$ and $a_1=2/\mu^2$ we have $c_1^{\pm}=1/\mu$ and
\[
    \tilde{g}(\alpha)=(1+(1/\mu)(\mu\alpha +\alpha^2/2))^{-b_1}=(1+\alpha/\mu)^{-2b_1} \ .
\]
Thus $g$ is the density of a Gamma$(1/\mu,2b_1)$ distribution.

\item \textbf{Example 2:} For $N=1$, $a_1=2/\mu^2$ and $b_1=1/2$ then $g(x)=\mu e^{-\mu x}$, the density of an exponentially distributed random variable with $\mu$.

\item \textbf{Example 3:} For $b_n=1$ and $a_n\geq 2/\mu^2,\ n\leq N$, using the equality $I_{1/2}(u)=2\sinh(u)/\sqrt{2\pi u}$, we obtain
\[
g(x)=2e^{-\mu x}\sum_1^N\frac{p_n\sinh(x\sqrt{\mu^2-2/a_n})}{a_n\sqrt{\mu^2-2/a_n}}\ .
\]

\item \textbf{Example 4:} For $\mu>1$, take $a_1=2,\ b_1=k/2$ so that $\tau_{X,\mu} \sim \chi^2(k)$ and $g$ is given by
\[
g(x)=\frac{\sqrt{\pi}e^{-\mu x}}{\Gamma(k/2)}\left(\frac{x}{2\sqrt{\mu^2-1}}\right)^{(k-1)/2}I_{(k-1)/2}(x\sqrt{\mu^2-1}) \ .
\]
\end{itemize}

While direct inversion of (\ref{eqn:matchlinearb>0}) could be complicated for a general density $f$, Theorem \ref{mdlinbound} provides a procedural approach for computing densities from the class $\CC$ and their corresponding matching densities given by (\ref{eqn:matchlinearmain})  by simply choosing the sequences $a_n,\ b_n,\ p_n$. Due to the restriction on the scale parameters (as in the class $\CC$) we could set a fixed scale parameter $a_n=1/c,\ c\leq \mu^2/2$ and choose a set of shape parameters $b_n$ and weights $p_n$ to match a particular density. In such cases, the unconditional density function $f$ becomes
\begin{eqnarray}
f(t)=ce^{-ct}\sum_1^{\infty}\frac{p_n}{\Gamma(b_n)}(ct)^{b_n-1}\ . \label{f:scaleconstant}
\end{eqnarray}
This class of densities includes the non-central $\chi^2(m)$ distribution by choosing $b_n=m/2+n,\ c=1/2,\ p_n=\frac{e^{-\delta^2/2}(\delta^2/2)^{n-1}}{(n-1)!}$ where $\delta$ is the non-central parameter. Some more general examples of distributions of the form (\ref{f:scaleconstant}) are given below.

\begin{itemize}
\item \textbf{Example 5:} $p_n=e^{-a}\frac{a^{n-1}}{(n-1)!},\ b_n=v+n,\ a_n=1/c$. Then $f$ and $g$ are given by:
\begin{align*}
f(t)&= \frac{c^{v/2+1}t^{v/2}e^{-ct-a}}{a^{v/2}}I_v(2\sqrt{act}) \\
&=\frac{ce^{-ct-a}}{a^v}\sum_{k=0}^{\infty}\frac{(act)^{v+k}}{k!\Gamma(v+k+1)}\ , \\
g(x)&=\sqrt{2\pi}c^{v+1}\left(\frac{x}{\sqrt{\mu^2-2c}}\right)^{v+1/2}\!\!\!e^{-\mu x-a}\sum_0^{\infty}\left(\frac{x c a}{\sqrt{\mu^2-2c}} \right)^k\frac{I_{v+k+1/2}(x\sqrt{\mu^2-2c})}{k! \Gamma(v+k-1)}\ .
\end{align*}
When $v=0,\ a=\frac{\alpha^2}{2\beta^2},\ c=\frac{1}{2\beta^2}$ then $f(t)=\frac{e^{-(t+\alpha^2)/(2\beta^2)}}{2\beta^2}I_0(\frac{\alpha\sqrt{t}}{\beta^2})$, so that if $\tau \sim f$ then $\sqrt{\tau}$ has Rice distribution with parameters $(\alpha,\beta)$.

\item \textbf{Example 6:} Suppose $b_n=n$ and $p_{n}=\frac{(\alpha_1)_{n-1}...(\alpha_r)_{n-1}}{(\beta_1)_{n-1}...(\beta_q)_{n-1}}/K$ where $(x)_{n}$ is the Pochhammer symbol (see (\ref{pochhammer}) in Appendix). Furthermore, we assume that the real valued sequences $\{\alpha_i\}_{i=1,...,r}$ and $\{\beta_j\}_{j=1,...,q}$ are such that $p_n>0$ for all $n\geq 1$ and $\sum_1^{\infty}p_n=K<\infty$. Then $f$ and $g$ are given by
\begin{align*}
f(t)&=\frac{ce^{-ct}}{K}\ _{r}F_q(\alpha_1,...,\alpha_r; \beta_1,...,\beta_q;ct)\\
&=\frac{ce^{-ct}}{K}\sum_0^{\infty}\frac{Kp_n(ct)^n}{n!}\ , \\
g(x)&=\sqrt{2 \pi c}e^{-\mu x}\sum_1^{\infty} \frac{p_n}{(n-1)!}\left(\frac{cx}{\sqrt{b^2-2c}}\right)^{n-1/2}I_{n-1/2}(x\sqrt{\mu^2-2c})\ ,
\end{align*}
where $_{r}F_q$ is the generalized hypergeometric series (see \citeN{GradshteynRyzhik00}, 9.14). When $r=q=1$ then $p_n=\frac{(\alpha)_{n-1}}{(\beta)_{n-1}}$. Furthermore, if $\alpha>0,\ \beta>\alpha+1$, then $p_n>0$ for all $n$ and the series $\sum_{n=1}^{\infty}p_n$ converge by Raabe's convergence test:
\[
\lim_{n\uparrow \infty}n(p_n/p_{n+1}-1)=\lim_{n\uparrow \infty}n(\beta-\alpha)/(\alpha+n)=\beta-\alpha>1\ .
\]
 In this case $f$ is given by
 \[
 f(t)=ce^{-ct}\ _1F_1(\alpha,\beta; ct)/K,\ K=\sum_1^{\infty}p_n\ ,
\]
  where $_{1}F_{1}$ is the confluent hypergeometric function of Kummer  (see \citeN{GradshteynRyzhik00}, 9.21).
\end{itemize}

We saw in Section 2 that for any strictly positive boundary, $b(t)>0, \forall t>0$, with corresponding unconditional density $f_b(t)$, the two densities $f_0$ and $f_b$ are related as in (\ref{eqn:unconditionaly0v0}) for all $t$. We will use this result to provide a connection between unconditional densities to the zero boundary and unconditional boundaries in the class $\CC$. To this end, let $\KK$ be the linear integral operator in (\ref{eqn:unconditionaly0v0}) with $b(t)=\mu t$, i.e.
\[
\KK := \frac{e^{-\frac{\mu s^2}{2(t-s)}}\mu s}{\sqrt{2\pi}(t-s)^{3/2}}\ .
\]
Define the new class of densities $\CC^{0}:= \KK \CC$, i.e. $f_0 \in \CC^{0}$ if there exists $f\in \CC$ (see \eqref{linearclass}) such that $f_0 = \KK f$. In this way, $\CC^0$ is a mapping which perturbs the probability measures of the class of densities $\CC$. Suppose $\tau_{X,b}\sim f \in \CC$, from Theorem \ref{mdlinbound} the matching density $g$ is given in \eqref{gamma}. Define the target density $f_0:=\KK f$ which is viewed as a target density for the zero boundary hitting time $\tau_{X,0}$.  By integrating equation \eqref{eqn:generaly0v0} (with $b(t)=\mu t$)  with respect to $g$, one finds that $g$ is a solution to the RFPT for the zero boundary and target density $f_0$. Therefore the pairs $(\mu t,\ f(t))$ and $(0,\ f_0(t))$ result in the same matching density $g$. This result can succinctly be stated as follows.
\begin{cor}
Let $f\in \CC$ and suppose $\tau_{X,0}\sim f_0:=\KK f \in \CC^0$. Then the matching density for the zero boundary and unconditional density $f_0$ is provided by (\ref{gamma}).
\end{cor}

Note that neither of the classes $\CC$ and $\CC^0$ includes the other. For example when $f$ is the density of the exponential distribution, there exists a unique $g$ which matches the pair $(\mu t, \ f)$, for $\mu>0$, since the class of Gamma distributions includes the exponential distribution. However, in Section 2 we saw that when the boundary is zero and the target density is exponential the unique solution to the RFPT Fredholm equation is $g(x)=\sqrt{2 \lambda}\sin(\sqrt{2\lambda}x)$ which is not a density function.


\section{Boundary Transformations}

Suppose there exists a random variable which is the solution to the RFPT for the boundary $b(t)$ and target density $f(t)$. The question that we address in this section is the following. {\it ``How does the target distribution change as the boundary is modified, while the distribution of $X$ is left unchanged?''} We will put particularly emphasis on the linear boundary case.

\subsection{Affine Transformations}

In this section we investigate affine modifications of the boundary. To this end,
denote the boundary specific stopping time
\begin{eqnarray}
\tau_{b(t)}:=\{t>0\ ; \ X+W_t \leq b(t) \} \ ,
\end{eqnarray}
where $X$ solves the RFPT for the pair $(b(t),f_{b(t)}(t))$ and has a density function $g$. Denoting the density of $\tau_{b(t)}$ by $f_{b(t)}$, from (\ref{eqn:mainreal}) (conditional on $X=x$) we obtain:
\begin{eqnarray*}
\int_0^{\infty}e^{-\alpha b(t)-t(\alpha \lambda +\alpha^2/2)}f_{b(t)+\lambda t}(t|x)dt\ =\ e^{-\alpha x}=\int_0^{\infty}e^{-\alpha b(t)-t\alpha^2/2}f_{b(t)}(t|x)dt\ .
\end{eqnarray*}
Multiplying by $g(x)$ and integrating over $x$, one finds the following connection between $f_{b(t)+\lambda t}$ and $f_{b(t)}$:
\begin{eqnarray}
\int_0^{\infty}e^{-\alpha b(t)-t(\alpha \lambda +\alpha^2/2)}f_{b(t)+\lambda t}(t)dt=\tilde{g}(\alpha) =\int_0^{\infty}e^{-\alpha b(t)-t\alpha^2/2}f_{b(t)}(t)dt \ .\label{eqn:generalAffineTrans}
\end{eqnarray}
This result relates a boundary specific transform of the density of $\tau_{b(t)+\lambda t}$ to a boundary specific transform of the density of $\tau_{b(t)}$ {\it and } the Laplace transform of the random starting point. For certain boundaries, the boundary specific transform is particular simple and we investigate the linear case next.

Equation \eqref{eqn:generalAffineTrans} reduces to a particularly simple expression when the boundary is linear. In particular, set $b(t)=\mu t$ and apply an affine transformation to change the slope from $\mu$ to $\nu$, i.e. choose $\lambda = (\nu-\mu)$, then \eqref{eqn:generalAffineTrans} takes the form
\begin{eqnarray}
\tilde{f}_{\nu t}(\alpha \nu+\alpha^2/2)=\tilde{g}(\alpha)=\tilde{f}_{\mu t}(\alpha \mu+\alpha^2/2)  \ . \label{newslope}
\end{eqnarray}
The above equality holds when $\alpha>\max(-\mu,-\nu)$ so that equation (\ref{eqn:matchlaplace}) is valid for both linear boundaries $\mu t$ and $\nu t$. Equation (\ref{newslope}), therefore allows one to represent the Laplace transform of the density of the FPT to the new slope in terms of the density of the FPT of the old slope.
\begin{lem}
Suppose $X$ solves the RFPT problem for the pair $(\mu t,f_{\mu t}(t))$. Then, keeping the distribution of $X$ unchanged, the stopping time of the randomized Brownian motion to the drift adjusted boundary $b(t) = \nu t$  has a density function $f_{\nu t}$ with a Laplace transform
\begin{eqnarray}
\tilde{f}_{\nu t}(s)=\tilde{f}_{\mu t}\left(\nu(\nu-\mu)+s+\sqrt{2}(\mu-\nu)\sqrt{s+\nu^2/2}\right),\quad s>0\ . \label{newslopelaplace}
\end{eqnarray}
\end{lem}
\begin{proof}
Let $s=\alpha \nu+\alpha^2/2$. Then $\alpha \mu +\alpha^2/2=\nu(\nu-\mu)+s+\sqrt{2}(\mu-\nu)\sqrt{s+\nu^2/2}$ and using (\ref{newslope}) we obtain the result. $\boxdot$
\end{proof}

Note that when $s>0$ the quantity in the brackets on the right side of (\ref{newslopelaplace}) is positive for all $\mu>0$ and $\nu \in \mathbb R$ and thus the right hand side exists for any density function $f_{b(t)}$. However, for $\mu<0$ (\ref{newslopelaplace}) is not a proper distribution.

Interestingly, equation  (\ref{newslopelaplace}) allows one to determine the distribution of the FPT to the new boundary without knowledged of the initial density $g$ of the starting point $X$. All that is required, is that $g$ exists and it solves the RFPT problem for the boundary $\mu t$ and the unconditional density $f_{\mu t}$. In the specific case when the density $f_{\mu t}$ is Gamma$(\alpha,\beta)$, the relation  (\ref{newslopelaplace}) then reads
\begin{align}
\tilde{f}_{\nu t}(s)&=(1/\alpha)^v\left(\frac{1}{\alpha}+\nu(\nu-\mu)+s+\sqrt{2}(\mu-\nu)\sqrt{s+\nu^2/2}\right)^{-\beta}
\nonumber \\
&=\left(A + \alpha s + B\sqrt{s+C} \right)^{-\beta} \ , \label{newslopegamma}
\end{align}
where $A=1+ \alpha\nu(\nu-\mu)$, $B=\alpha\sqrt{2}(\mu-\nu)$, and $C=\nu^2/2$. In Figure \ref{fig:Perburb}, we numerically invert the Laplace transform and plot $f_{\nu t}$ when $\mu=1$ for a few choices of $\nu$. This idea of changing drifts has potential applications in the context of finance and insurance, where such a drift change corresponds to changing probability measures from the historical measure to a valuation measure.

As a final result for affine transformations, through Lemma \ref{laplacegammalemma} in Appendix A, we find the following explicit representation for $f_{\nu t}$ when $\mu > \nu$:
\begin{align*}
f_{\nu t}(t)
&=\frac{(B/\alpha)\alpha^{-\beta}e^{-tC}}{2\Gamma(\beta)\sqrt{\pi}}\int_0^t(t-x)^{-3/2} x^\beta \exp\left\{-x(-C+A/\alpha)-\tfrac{x^2(B/\alpha)^2}{4(t-x)}\right\}\ dx\ .
\end{align*}
This result may be more useful than the laplace transform representation \eqref{newslopegamma} in certain cases.

\subsection{Scaling of Boundaries}

In this section we analyze transformations induced by a scaling of the boundary. In particular, we assume the boundary $b_{\lambda}(t)$ is indexed by a parameter $\lambda>0$ and satisfies the following scaling property
\begin{eqnarray}\label{scalingprop}
b_{\lambda}(t)=b_1(\lambda^2 t)/\lambda\ .
\end{eqnarray}
Note that the linear boundaries $b_\lambda(t) = \lambda t$ satisfies this scaling property, further, scaling corresponds to modifying the slope. We are interested in how the stopping times
\[
\tau_{X,\lambda}=\inf \{t>0; X+W_t \leq  b_{\lambda}(t)\},\ X \geq 0\ ,
\]
indexed by the scale parameter, are related to one another.  We will see that, through the scaling property of Brownian motions, we can investigate randomizing over the scale parameter rather than the starting point of the Brownian motion. Such a randomization can be viewed as a generalization of randomizing the slope of a linear boundary. In the RFPT problem we encountered the class $\CC$ (see \eqref{linearclass}) of mixtures of Gamma distributions. This class has a restriction on the minimum scale parameter in relation to the slope of the linear boundary. Thus, it is plausible that the randomization over the boundaries scale parameter may allow this restriction to be removed or at least modified. In the end, we hope to be able to probe a larger class of target densities. In light of these comments, we define the following modification to the RFPT problem.
\begin{defn}[Randomized Scaling First Passage Time Problem (RSFPT)]
Given a scale class of boundaries $b_\lambda\!:\![0,\infty)\rightarrow \RR$, and a probability measure $\mu$ on $[0,\infty)$, find a random variable $X$ such that $\mu$ is the law of the randomized scaling first passage time $\tau_{1,\lambda X}$.
\end{defn}

To solve this problem we assume, as usual, that the boundary is regular in the sense that $P(\tau_{X,\lambda}=0)=0$ for all values of $\lambda>0$ and $X$. As a consequence of the scaling properties of Brownian motion we have the following easy Lemma.
\begin{lem}\label{randomparameter}
If $X$ has no probability mass at zero, then
\begin{align}
\tau_{1,\lambda X}\stackrel{d}{=}\frac{1}{X^2}\tau_{X,\lambda} \ .
\label{eqn:randomScale}
\end{align}
\end{lem}
\begin{proof}
First, conditioning on $X\ne 0$, and using the scaling property of Brownian motions and the class of boundaries, we have
\begin{align*}
\frac{\tau_{X,\lambda}}{X^2}
&= \inf \left\{t/X^2>0 \ ; \ X+W_t \leq  b_{\lambda}(t) \right\}\\
&= \inf \left\{t/X^2>0\ ; \ 1+W_t/X \leq  b_{1}(\lambda^2 t)/(\lambda X) \right\}\\
&\stackrel{d}{=}  \inf \left\{u>0 \ ; \ 1+W_u \leq  b_{1}(\lambda^2 X^2 u)/(\lambda X) \right\}\\
&= \inf \left\{u>0\ ; \ 1+W_u \leq  b_{\lambda X}(u) \right\}\\
&=\tau_{1,\lambda X} \ .
\end{align*}
Since the equality holds for every $X\ne 0$, and $X$ has no mass at $0$, we have the result. $\boxdot$
\end{proof}

Suppose we have a solution $g_\lambda(x)$ to the RFPT for the boundary/density pair $(b_\lambda(t), f_\lambda(t))$, i.e. $\tau_{X,\lambda}$ has density $f_\lambda$ when $X$ has density $g_\lambda$. Then, through Lemma \ref{randomparameter}, the r.h.s. of \eqref{eqn:randomScale} is known, and therefore the randomized scaling first passage time $\tau_{1,\lambda X}$ is known and we therefore have a solution to the RSFPT.

Lemma \ref{randomparameter} provides a solution to another related FPT problem where the intercept and the slope are both randomized simultaneously. In particular, consider the stopping time $\tau_{X,\frac{a}{X}}$ (for $a>0$) with $X$ having no probability mass at zero -- we can call the problem of finding the distribution of $X$ given the distribution of $\tau_{X,\frac{a}{X}}$ the {\it randomized double scaling FPT problem (RDSFPT)}. To solve it, consider the three random variables $\ln X$, $\ln \tau_{X,\frac{a}{X}}$ and $\ln \tau_{1,a}$ and denote their mgf's by $\tilde g$, $\tilde f$ and $\tilde h$, respectively. Then we have the following result.
\begin{cor}\label{cor:random}
If $X$ has no probability mass at zero, then 
\begin{equation}
\tilde{g}(2\alpha)=\frac{\tilde{f}(\alpha)}{\tilde{h}(\alpha)}\label{eqn:random2}
\end{equation}
\end{cor}
\begin{proof}
From Lemma \ref{randomparameter} we have $\tau_{1,a}\stackrel{d}{=}\frac{1}{X^2} \tau_{X,\frac{a}{X}}$. Consequently, $\log(\tau_{X,\frac{a}{X}})\stackrel{d}{=}  2\log (X)+\log(\tau_{1,a})$; furthermore, $\log(X)$ is independent of $\log(\tau_{1,a})$. This independence neatly implies \eqref{eqn:random2} for the mgf's. $\boxdot$
\end{proof}
Corollary \ref{cor:random} implies that if we know the mgf of $X$ then we have explicitly the mgf of $\tau_{X,\frac{a}{X}}$ and vice versa -- assuming the mgf of $\tau_{1,a}$ is known. If those mgfs are invertible we have completely solved the RDSFPT.

\section{Conclusions}

In this article we introduced a problem related to both the forward and inverse first passage time problems. The new problem, coined the randomized first passage time problem (RFPT), seeks the density $g$ of the initial starting point $X$ of a Brownian motion such that the density of the FPT $\tau_X$ to a fixed boundary $b$ is a given target function $f$. We prove two uniqueness results (see Prop. \ref{prop:hermitetransform} and \ref{prop:laplacetransform}) and provide an existence result (see Thm. \ref{existence}) for $g$. Furthermore, when the boundary is linear we provide several explicit examples of the density $g$ based on the target $f$ being a mixture of gamma distributions. Finally, we provide examples of how the target densities for boundaries in an affine and scaling class are related to one another.

There are a number of theoretical directions left open. The first being making the existence result easier to check for boundaries other than the linear and square-root boundaries. The second is making the connection to Skorohod's embedding problem more explicit. We have already made progress on this direction and will report the results in another article. A third theoretical direction is to demonstrate the consistency of the existence results in Propsitions \ref{prop:hermitetransform} and \ref{prop:laplacetransform}.

These results have several applications ranging from finance, to cosmology, to biology. In the financial context, the Brownian motion can represent the log of the leverage ratio of a company. When this leverage ratio hits a critical barrier (the boundary), the company defaults. The distribution of the default time is observable through yields of certain financial instruments (such as bonds and credit default swaps). Our results then provide a methodology for obtaining the distribution of initial log leverage ratio consistent with market prices. As well, the affine and scaling transformations allows one to change measure from statistical to risk-neutral (or pricing) measures in a natural manner.

\section{Acknowledgements}

This work is supported in part by NSERC and MITACS.

\appendix
\section{Proofs of Results}
\textbf{Proof of Lemma \ref{prop:hermitetransform}:}\
Suppose $b(t)>0$ for some $t>0$ so that for $y=-x$ the condition $y=-x<b(t)-x$ is satisfied. Then the kernel of (\ref{eqn:general1}) becomes free of $x$. Multiply (\ref{eqn:general1}) by $g(x)\in G$, assuming that there exists such a $g$ which solves (\ref{eqn:matchmain}), and integrate $x$ on $(0,\infty)$ to obtain:
\begin{eqnarray}
\int_0^{\infty}\frac{e^{-\frac{x^2}{2t}}H_n(x/\sqrt{2t})}{t^{(n+1)/2}}g(x)dx=\int_0^{\infty}\int_0^t\frac{e^{-\frac{b(s)^2}{2(t-s)}}}{(t-s)^{\frac{n+1}{2}}}H_n\left(\frac{b(s)}{\sqrt{2(t-s)}}\right)f(s|x)g(x)dsdx \label{eqn:hermite1}
\end{eqnarray}
Next we examine the right hand side of (\ref{eqn:hermite1}) and justify the exchange of the order of integration by the use of Fubini's theorem after we show the quantity under the double integral is absolutely integrable. Thus,using (\ref{hermiteinequality}), we have:
\begin{eqnarray}
&&\int_0^{\infty}\int_0^te^{-\frac{(b(s))^2}{2(t-s)}}\frac{|H_n\left(\frac{b(s)}{\sqrt{2(t-s)}}\right)|}{(t-s)^{(n+1)/2}}f(s|x)|g(x)|dsdx\\
&\leq& z_n\int_0^{\infty}\int_0^te^{-\frac{(b(s))^2}{2(t-s)}}\frac{e^{q_n\frac{b(s)}{\sqrt{(t-s)}}}}{(t-s)^{(n+1)/2}}f(s|x)|g(x)|dsdx\\
&=&z'_n\int_0^{\infty}\int_0^t\frac{e^{-\frac{1}{2}\left(\frac{b(s)}{\sqrt{(t-s)}}-q_n\right)^2}}{(t-s)^{(n+1)/2}}f(s|x)|g(x)|dsdx \label{eqn:proofhermite}
\end{eqnarray}
where $q_n=\sqrt{2\left\lfloor n/2 \right\rfloor}$ and $z_n=2^{n/2-\left\lfloor n/2\right\rfloor}(n!/\left\lfloor n/2\right\rfloor!)$ and $z'_n=2^{n/2}(n!/\left\lfloor n/2\right\rfloor!)$. Let $p_n(s,t)=\frac{e^{-\frac{1}{2}(\frac{b(s)}{\sqrt{(t-s)}}-q_n)^2}}{(t-s)^{(n+1)/2}}$. We have $0\leq p_n(0,t)<t^{-(n+1)/2}$ and $p_n(t,t)=0$ since $b$ is continuous. Thus $p_n(.,s)$ is continuous and since it is finite at $0$ and at $t$ then it is bounded
on the interval $[0,t]$ by, say, $M_n(t)$. Then, continuing from (\ref{eqn:proofhermite}) we have:
\begin{eqnarray*}
z'_n\int_0^{\infty}\int_0^t\frac{e^{-\frac{1}{2}\left(\frac{b(s)}{\sqrt{2(t-s)}}-q_n\right)^2}}{(t-s)^{(n+1)/2}}f(s|x)g(x)dsdx&\leq& z'_nM_n(t)\int_0^{\infty}\int_0^tf(s|x)|g(x)|dsdx\\
&=&z'_nM_n(t)\int_0^{\infty}F(t|x)|g(x)|dx
\end{eqnarray*}
where $F(t|x)$ is the conditional cdf. Since $b$ is continuous, for each $t>0$, we can find an $k(t)<0$ and $N(t)\geq -k(t)$ such that $b(s)-x<-k(t)-x<0$ for all $s\leq t$ and $x\geq N(t)$. Therefore, for $x\geq N(t)$, $F(t|x) \leq 2\Phi\left(-\frac{x+k(t)}{\sqrt{t}} \right)\leq \frac{\sqrt{t}e^{-(x+k(t))^2/(2t)}}{x+k(t)}$ and continuing from the last equality we obtain:
\begin{eqnarray*}
&&z'_nM_n(t)\int_0^{\infty}F(t|x)|g(x)|dx=\\
&&=z'_nM_n(t)\left[\int_0^{N(t)}F(t|x)|g(x)|dx+\int_{N(t)}^{\infty}F(t|x)|g(x)|dx\right]\\
&&\leq z'_nM_n(t)\left[\int_0^{N(t)}|g(x)|dx+\sqrt{t} \int_{N(t)}^{\infty}\frac{e^{-(x+k(t))^2/(2t)}}{x+k(t)}|g(x)|dx\right]
\end{eqnarray*}
The integrals in the square brackets are both finite for all $g\in G$ (and more generally for all $g\sim O(e^{x^r}), 0\leq r<2$ for large $x$ and which are absolutely integrable in the neighborhood of zero) and hence we can exchange the order of integration and the right side of (\ref{eqn:hermite1}) becomes $a_n(t)$. Furthermore, the left side of (\ref{eqn:hermite1}) can be written as
\[
\frac{\sqrt{2}}{t^{n/2}}\int_0^{\infty}e^{-u^2}H_n(u)g(u\sqrt{2t})du\ ,
\]
which is the Hermite transform of $g(u\sqrt{2t})$. Thus, since the Hermite functions form a complete orthogonal basis in the space $L^2(e^{-u^2})$, if the Fredholm equation (\ref{eqn:matchmain}) has a solution $g\in G\bigcap L^2(e^{-u^2})$, it is unique (by the uniqueness of the Hermite transform) and it is given by (\ref{eqn:hermitetransform}) provided that $b$ is continuous and $b(t)>0$ for some $t$. $\boxdot$

\noindent \textbf{Proof of Lemma \ref{completemonotone}:}\
Let $b$ be a continuous function. Using (\ref{hermitedefinition}) and (\ref{hermiteinequality}) we have
\begin{eqnarray*}
|\frac{d^n}{d\alpha^n}\exp\{-(\alpha\sqrt{t/2}+b(t)/\sqrt{2t})^2 \}|&=&|(-\sqrt{t/2})^ne^{-\left(\frac{b(t)+\alpha t}{\sqrt{2t}}\right)^2}H_n\left(\frac{b(t)+\alpha t}{\sqrt{2t}}\right)|\\
&\leq& K_nt^{n/2}e^{-\left(\frac{b(t)+\alpha t}{\sqrt{2t}} \right)^2+2\sqrt{\left\lfloor n/2 \right\rfloor}\frac{b(t)+\alpha t}{\sqrt{2t}}}
\end{eqnarray*}
where $K_n=2^{-\left\lfloor n/2\right\rfloor}(n!/\left\lfloor n/2\right\rfloor!)$. Thus
\begin{eqnarray*}
\int_0^{\infty}|\frac{d^n}{d\alpha^n}e^{-\alpha b(t)-\alpha^2t/2}|f(t)dt&=&\int_0^{\infty}|\frac{d^n}{d\alpha^n}e^{-\left(\frac{b(t)+\alpha t}{\sqrt{2t}}\right)^2 }|e^{b^2(t)/(2t)}f(t)dt\\
&\leq &K_n\int_0^{\infty}t^{n/2}e^{-\alpha b(t)-\alpha^2t/2+2\sqrt{\left\lfloor n/2 \right\rfloor}\frac{b(t)+\alpha t}{\sqrt{2t}}}f(t)dt
\end{eqnarray*}
Let $\epsilon_n$ be such that $t^{n/2}e^{-\alpha b(t)-\alpha^2t/2+2\sqrt{\left\lfloor n/2 \right\rfloor}\alpha\sqrt{t/2}}<1$ for all $t<\epsilon_n$. Then the last integral is bounded by
\begin{eqnarray*}
&&K_n\int_0^{\epsilon_n}e^{\sqrt{2\left\lfloor n/2 \right\rfloor}b(t)/\sqrt{t}}f(t)dt+K_n\int_{\epsilon_n}^{\infty}t^{n/2}e^{-\alpha b(t)-\alpha^2t/2+2\sqrt{\left\lfloor n/2 \right\rfloor}\frac{b(t)+\alpha t}{\sqrt{2t}}}f(t)dt
\end{eqnarray*}
Since $b$ satisfies $\lim_{t\uparrow \infty}(b(t)+\alpha t)>-\infty$ then if $\lim_{t \uparrow \infty}b(t) <0$, it follows that $b(t)\sim -t^{\gamma},\ \gamma<1$ for large $t$ and the power in the exponent of the second integral above is dominated by $-\alpha^2 t/2$. Similarly, if $\lim_{t \uparrow \infty}b(t) \geq 0$ then $\alpha b(t)$ dominates $b(t)/\sqrt{2t}$. Therefore, the second integral is finite for all $\alpha>0,\ n\geq 1,$ since $$\lim_{t \uparrow \infty}t^{n/2}e^{-\alpha b(t)-\alpha^2t/2+2\sqrt{\left\lfloor n/2 \right\rfloor}\frac{b(t)+\alpha t}{\sqrt{2t}}}=0$$ and $b$ is continuous. Under the hypothesis of Lemma \ref{completemonotone} it follows that
\begin{eqnarray*}
\int_0^{\infty}|\frac{d^n}{d\alpha^n}e^{-\alpha b(t)-\alpha^2t/2}|f(t)dt<\infty
\end{eqnarray*}
and therefore
\begin{eqnarray*}
(-1)^n\frac{d^n}{d\alpha^n}r(\alpha)&=& (-1)^n\int_0^{\infty}\frac{d^n}{d\alpha^n}e^{-\alpha b(t)-\alpha^2t/2}f(t)dt\\
&=&\int_0^{\infty}\left(\frac{t}{2}\right)^{n/2}e^{-\alpha b(t)-\alpha^2t/2}H_n\left(\frac{b(t)+\alpha t}{\sqrt{2t}}\right)f(t)dt
\end{eqnarray*}
Complete monotonicity of $r(\alpha)$ is thus equivalent to the last integral being nonnegative for all $\alpha>0,\ n\geq 1$. This completes the proof. $\boxdot$\\

\noindent \textbf{Proof of Corrolary \ref{linboundexist}:}\
When $b(t)=\mu t$ and $f(t)=\frac{t^{b-1}e^{-t/a}}{\Gamma(b)(a)^b},\ \mu,a,b>0,$ the hypothesis of Lemma \ref{completemonotone} hold and thus we only need to check the complete monotonicity of $r(\alpha)$ by showing that (\ref{monotonecheck}) is nonnegative for all $\alpha>0$ and $n \geq 1$. In this case (\ref{monotonecheck}) becomes
\begin{eqnarray*}
&&\int_0^{\infty}t^{n/2}e^{-t(\alpha^2+2\alpha \mu)/2}H_n\left(\frac{(\alpha+\mu)\sqrt{t}}{\sqrt{2}}\right)\frac{t^{b-1}e^{-t/a}}{\Gamma(b)a^b}dt=\\
&&=\frac{2}{\Gamma(b)a^b}\int_0^{\infty}x^{n+2b-1}e^{-x^2(\alpha^2+2\alpha \mu+2/a)/2}H_n\left(\frac{(\alpha+\mu)x}{\sqrt{2}}\right)dx
\end{eqnarray*}
The last quantity is strictly positive for $n=1$ and all $\alpha>0$ since $H_1(z)=2z$. Thus, set $n=2k+\delta,\ k\geq 1,$ where $\delta$ is $0$ or $1$. Then, using \citeN{Prudnikov86} (2.20.3(4)), we obtain:
\begin{eqnarray*}
&&\frac{2}{\Gamma(b)a^b}\int_0^{\infty}x^{2k+\delta+2b-1}e^{-x^2(\alpha^2+2\alpha \mu+2/a)/2}H_{2k+\delta}\left(\frac{(\alpha+\mu)x}{\sqrt{2}}\right)dx=\\
&&=\frac{2}{\Gamma(b)a^b}C_k(\alpha)(-1)^k(1/2-k-b)_k\Gamma(k+\delta+b)\times\\
&&\times _{2}F_1\left(k+b+\delta/2,k+b+(\delta+1)/2; b+1/2;\frac{\mu^2-2/a}{(\mu+\alpha)^2} \right)=\\
&&=\frac{2}{\Gamma(b)a^b}C_k(\alpha)\Gamma(k+\delta+b)(-1)^{2k}(k+b-1/2)(k+b-1/2-1)...\times\\
&&\times(k+b-1/2-(k-1)) _{2}F_1\left(k+b+\frac{\delta}{2},k+b+\frac{\delta+1}{2}; b+\frac{1}{2};\frac{\mu^2-2/a}{(\mu+\alpha)^2} \right)
\end{eqnarray*}
where $(.)_k$ is the Pochhammer symbol (see (\ref{pochhammer}) below), $_{2}F_1$ is the Hypergeometric function (see (\ref{hgfunction}) below) and $C_k(\alpha)=\frac{2^{3k+b-1+3\delta/2}}{(\alpha+\mu)^{2k+2b+\delta}}$. Thus, we see that (\ref{monotonecheck}) is nonnegative whenever $_{2}F_1 \geq 0$. This is the case when $\mu^2-2/a \geq 0$. Finally, the existence of $X$ is guaranteed by Theorem \ref{existence}. This completes the proof. $\boxdot$\\
\begin{lem}
\label{laplacegammalemma}
Suppose $v,d,B,C>0$ are positive constants and $A\in \mathbb R$ and such that $d+s+A+B\sqrt{C+s}>0$ for any $s>0$. Then the Laplace transform of $$h(t):=\frac{B}{2\Gamma(v)\sqrt{\pi}}\int_0^t(t-x)^{-3/2}x^ve^{-C(t-x)-x(d+A)-\frac{x^2B^2}{4(t-x)}}dx $$ is given by $$\tilde{h}(s)=\left(d+A+s+B\sqrt{s+C} \right)^{-v}$$ for $s>0$.
\end{lem}
\begin{proof} We compute the Laplace transform directly:
\begin{eqnarray*}
\tilde{h}(s)&=&\int_0^{\infty}e^{-st}h(t)dt=\frac{B}{2\Gamma(v)\sqrt{\pi}}\int_0^{\infty}x^ve^{-x(d+A+s)}\int_0^{\infty}u^{-3/2}e^{-u(C+s)-\frac{x^2B^2}{4u}}dudx\\
&=&\frac{B}{\Gamma(v)\sqrt{\pi}}\int_0^{\infty}x^ve^{-x(d+A+s)}\left(\frac{x^2B^2}{4(C+s)}\right)^{-1/4}K_{-1/2}\left(xB\sqrt{C+s} \right)dx\\
&=&\frac{B}{2\Gamma(v)\sqrt{\pi}}\frac{2\sqrt{\pi}}{B}\int_0^{\infty}x^{v-1}e^{-x(d+A+s+B\sqrt{C+s})}dx\\
&=&\left(d+A+s+B\sqrt{s+C} \right)^{-v}  \boxdot
\end{eqnarray*}
where we have used equation (3.471(9)) from \citeN{GradshteynRyzhik00} in the second line above.
\end{proof}

\section{Special Functions}

In this Appendix, we collect several useful formulae concerning the parabolic cylinder, Hermite and Hypergeometric functions.

\subsection{Parabolic Cylinder Function} \label{sec:ParabolicCylinder}
\label{pcf}

The parabolic cylinder function can be characterized as the solutions to the following differential equation indexed by $p$
\begin{eqnarray}
\frac{d^2u}{dz^2}+\left(p+1/2-\frac{z^2}{4}\right)u=0 \ . \label{diffeqn1}
\end{eqnarray}
The solutions are $u=D_p(z),\ D_p(-z),\ D_{-p-1}(iz),\ D_{-p-1}(-iz)$. These functions also admit and integral representation for $p<0$:
\begin{eqnarray}
D_p(z)=\frac{e^{-z^2/4}}{\Gamma(-p)}\int_0^{\infty}e^{-xz-x^2/2}x^{-p-1}dx \ . \label{pcfintegral}
\end{eqnarray}
As well, the parabolic cylinder functions for integer parameters are related to the Hermite polynomials
\begin{eqnarray}
&&D_n(z)=2^{-n/2}e^{-z^2/4}H_n\left(\frac{z}{\sqrt{2}}\right) \ . \label{pcfHn}
\end{eqnarray}
Here $H_n$ is the Hermite polynomial of degree $n$.

\subsection{Hermite Polynomials}

The Hermite polynomials are defined as the re-scaled $n$-fold derivative of the gaussian density as follows
\begin{equation}
H_n(x)=(-1)^n e^{x^2}\frac{d^n}{dx^n}\left(e^{-x^2} \right)\ .\label{hermitedefinition}
\end{equation}
Some useful bounds on the Hermite polynomials are
\begin{equation}
|H_n(x)| \leq 2^{n/2-\left\lfloor n/2\right\rfloor}(n!/\left\lfloor n/2\right\rfloor!) e^{2x\sqrt{\left\lfloor n/2\right\rfloor}} \ . \label{hermiteinequality}
\end{equation}
Finally, the Hermite polynomials form a complete basis on the Hilbert space $L^2(\mathbb R, e^{-x^2})$ and satisfy the orthogonality relation
\begin{equation}
\int_{-\infty}^{\infty}H_m(x)\frac{H_n(x)}{\sqrt{\pi}2^n n!}=\delta_n(m)\ .
\end{equation}
Here $\delta_n(m)=1$ if $m=n$ and zero otherwise.

\subsection{Hypergeometric series}

The Hypergeometric function is defined as:
\begin{eqnarray}
_{2}F_1(a,b;c;x)=1+\frac{ab}{c.1}x+\frac{a(a+1)b(b+1)}{c(c+1).1.2}x^2+...=\sum_{n=0}^{\infty}\frac{(a)_n(b)_n x^n}{(c_n)n!} \ , \label{hgfunction}
\end{eqnarray}
where the Pochhammer symbol $(z)_n$ denotes
\begin{eqnarray}
(z)_n:=z(z+1)...(z+n-1),\ (n=1,2,...),\ (z)_0=1 \ . \label{pochhammer}
\end{eqnarray}

\bibliographystyle{chicago}
\bibliography{matchden_paperbib}

\end{document}